\documentclass[a4paper,11pt]{article}
\usepackage[T1]{fontenc}
\usepackage[utf8]{inputenc}
\usepackage{amsmath,amsthm,amssymb,color,graphicx,bm,mathtools}
\usepackage{algorithm}
\usepackage{algorithmic}
\usepackage{layaureo}
\usepackage{microtype}
\usepackage{emptypage}
\usepackage{indentfirst}
\usepackage{CJKutf8}
\usepackage{enumitem}

\DeclarePairedDelimiter{\abs}{\lvert}{\rvert}
\DeclarePairedDelimiter{\norm}{\lVert}{\rVert}

\usepackage[all, pdf]{xy}
\usepackage{subfig}
\usepackage{enumitem}
\usepackage{color}
\usepackage{graphicx,booktabs}
\usepackage[bookmarks,colorlinks=true,linkcolor=blue]{hyperref} 
\newcommand{\numberset}{\mathbb}
\newcommand{\R}{\numberset{R}} 
\renewcommand{\phi}{\varphi} 
\renewcommand{\chi}{\mathcal{X}} 
\newcommand{\Hst}{\tilde{H}^s}
\renewcommand{\epsilon}{\varepsilon}

\newcommand{\weak}{\rightharpoonup}
\newcommand{\weakstar}{\rightharpoonup^*}

\newtheorem{theorem}{Theorem}
\newtheorem{prop}[theorem]{Proposition}
\newtheorem{definition}[theorem]{Definition}

\newtheorem{conj}{Conjecture}

\newtheorem{remark}[theorem]{Remark}

\newenvironment{system}%
{\left\lbrace\begin{aligned}}%
{\end{aligned}\right.}

\begin{document}

\title{A variational scheme for hyperbolic obstacle problems}
\author{M. Bonafini\thanks{Dipartimento di  Matematica, Universit\`a di Trento, Italy, e-mail: mauro.bonafini@unitn.it}, M. Novaga\thanks{Dipartimento di Matematica, Universit\`a di Pisa, Italy, e-mail: matteo.novaga@unipi.it}, G. Orlandi\thanks{Dipartimento di Informatica, Universit\`a di Verona, Italy, e-mail: giandomenico.orlandi@univr.it}}
\date{\today}

\maketitle

\begin{abstract}
	We consider an obstacle problem for (possibly non-local) wave equations, and we prove existence of weak solutions through a convex minimization approach based on a time discrete approximation scheme. We provide the corresponding numerical implementation and raise some open questions.
\end{abstract}

\section{Introduction}

Obstacle type problems are nowadays a well established subject with many dedicated contributions in the recent literature. Obstacle problems for the minimizers of classical energies and regularity of the arising free boundary have been extensively studied, both for local operators (see, e.g. \cite{Ca, Ro18} and references therein) and non-local fractional type operators (see, e.g. \cite{Si} and the review \cite{Ro18}). The corresponding evolutive equations have also been considered, mainly in the parabolic context \cite{CaPeSh, CaFi13, NoOk15, BaFiRo18}. What seems to be missing in the picture is the hyperbolic scenario which, despite being in some cases as natural as the previous ones, has received little attention so far.

Among the available results for hyperbolic obstacle problems there is a series of works by Schatzman and collaborators \cite{Schatzman78, Schatzman80, Schatzman81, PaoliSchatzman02I}, where the existence of a solution is proved via penalty methods and, furthermore, existence of energy preserving solutions are proved in dimension $1$ whenever the obstacle is concave \cite{Schatzman80}. The problem is also considered in \cite{Maruo85}, where the author proves the existence of a (possibly dissipative) solution within a more general framework but under technical hypotheses. More recently the $1$d situation has been investigated in \cite{Ki09} through a minimization approach based on time discretization, see also \cite{GiSv09, OmKaNa09, Ta94, DaLa11} for contributions on related problems using the same point of view.
Another variational approach to hyperbolic problems, through an elliptic regularization suggested by De Giorgi, is given in \cite{SeTi} and subsequent papers (see for instance \cite{DaDe} for time dependent domains).

In this paper we use a convex minimization approach, relying on a semi-discrete approximation scheme (as in \cite{Ki09, GiSv09, DaLa11}), to deal with more general situations so as to include also non-local hyperbolic problems in the presence of obstacles, in arbitrary dimension. As main results we prove existence of a suitably defined weak solution to the wave equation involving the fractional Laplacian with or without an obstacle, together with the corresponding energy estimates. Those results are summarized in Theorem \ref{thm:main1} and Theorem \ref{thm:main2} (see Section \ref{sec:free} and \ref{sec:obstacle}). The approximating scheme allows to perform numerical simulations which give quite precise evidence of dynamical effects. In particular, based on our numerical experiments for the obstacle problem, we conjecture that this method is able to select, in cases of nonuniqueness, the most dissipative solution, that is to say the one losing the maximum amount of energy at contact times.

Eventually, we remark that this approach is quite robust and can be extended for instance to the case of adhesive phenomena: in these situations an elastic string interacts with a rigid substrate through an adhesive layer \cite{CocliteFlorioLigaboMaddalena17} and the potential energy governing the interaction can be easily incorporated in our variational scheme.

The paper is organized as follows. We first recall the main properties of the fractional Laplace operator and fractional Sobolev spaces in Section \ref{sec:fractional} and then, in Section \ref{sec:free}, we introduce the time-disretized variational scheme and apply it to the non-local wave equation (with the fractional Laplacian), proving Theorem \ref{thm:main1}. In Section \ref{sec:obstacle} we adapt the scheme so as to include the obstacle problem, proving existence of weak solutions in Theorem \ref{thm:main2}. In the last section we describe the corresponding numerical implementation providing some examples and we conclude with some remarks and open questions.

\section{Fractional Sobolev spaces and the fractional Laplacian operator}\label{sec:fractional}
In this section we briefly review the main definitions and properties of the fractional setting and we fix the notation used in the rest of the paper. For a more complete introduction to fractional Sobolev spaces we point to \cite{DiNezzaPalatucciValdinoci12, mclean2000strongly} and references therein.

\textbf{Fractional Sobolev spaces.} Let $\Omega \subset \R^d$ be an open set. For $s \in \R$, we define the Sobolev spaces $H^s(\Omega)$ as follows:
\begin{itemize}
	\item for $s \in (0,1)$ and $u \in L^2(\Omega)$, define the Gagliardo semi-norm of $u$ as
	\[
	[u]_{H^s(\Omega)} = \left( \int_\Omega \int_\Omega \frac{\abs{u(x)-u(y)}^2}{\abs{x-y}^{d+2s}} \, dxdy \right)^{\frac{1}{2}}.
	\]
	The fractional Sobolev space $H^s(\Omega)$ is then defined as
	\[
	H^s(\Omega) = \left\{ u \in L^2(\Omega) \,:\, [u]_{H^s(\Omega)} < \infty \right\},
	\]
	with norm $||u||_{H^s(\Omega)} = (||u||_{L^2(\Omega)}^2 + [u]_{H^s(\Omega)}^2)^{1/2}$;
	\item for $s \geq 1$ let us write $s = [s] + \{s\}$, with $[s]$ integer and $0\leq \{s\} < 1$. The space $H^s(\Omega)$ is then defined as
	\[
	H^s(\Omega) = \{ u \in H^{[s]}(\Omega) \,:\, D^\alpha u \in H^{\{s\}}(\Omega) \text{ for any } \alpha \text{ s.t. } |\alpha| = [s] \},
	\]
	with norm $||u||_{H^s(\Omega)} = (||u||_{H^{[s]}(\Omega)}^2+\sum_{|\alpha|=[s]}||D^\alpha u||_{H^{\{s\}}(\Omega)}^2)^{1/2}$;
	\item for $s < 0$ we define $H^s(\Omega) = (H^{-s}_0(\Omega))^*$, where as usual the space $H^s_0(\Omega)$ is obtained as the closure of $C^\infty_c(\Omega)$ in the $||\cdot||_{H^s(\Omega)}$ norm.
\end{itemize}

\textbf{Fractional Laplacian.} For any $s > 0$, denote by $(-\Delta)^s$ the fractional Laplace operator, which (up to normalization factors) can be defined as follows:
\begin{itemize}
	\item for $s \in (0,1)$, we set
	\[
	-(-\Delta)^s u(x) = \int_{\R^d} \frac{u(x+y)-2u(x)+u(x-y)}{\abs{y}^{d+2s}} \, dy, \quad x \in \R^d;
	\]
	\item for $s \geq 1$, $s = [s] + \{s\}$, we set $(-\Delta)^s = (-\Delta)^{\{s\}} \circ (-\Delta)^{[s]}$.
\end{itemize}
Let us define, for any $u, v \in H^s(\R^d)$, the bilinear form
\[
[u, v]_{s} = \int_{\R^{d}} (-\Delta)^{s/2}u(x) \cdot  (-\Delta)^{s/2}v(x) \, dx
\]
and the corresponding semi-norm $[u]_s = \sqrt{[u,u]_s} = ||(-\Delta)^{s/2}u||_{L^2(\R^d)}$. Define on $H^s(\R^d)$ the norm $||u||_s = (||u||_{L^2(\R^d)}^2 + [u]_s^2 )^{1/2}$, which in turn is equivalent to the norm $||\cdot||_{H^s(\R^d)}$.

\textbf{The spaces $\Hst(\Omega)$.} Let $s > 0$ and fix $\Omega$ to be an open bounded set with Lipschitz boundary. The space we are going to work with throughout this paper is
\[
\tilde{H}^s(\Omega) = \{ u \in H^s(\R^d) \,:\, u = 0 \text{ a.e. in } \R^d\setminus \Omega \},
\]
endowed with the $||\cdot||_s$ norm. This space corresponds to the closure of $C^\infty_c(\Omega)$ with respect to the $||\cdot||_s$ norm.
We have also $(\Hst(\Omega))^* = H^{-s}(\Omega)$, see \cite[Theorem 3.30]{mclean2000strongly}.

We finally recall the following embedding results (see \cite{DiNezzaPalatucciValdinoci12}).

\begin{theorem}
	Let $s > 0$. The following holds:
	\begin{itemize}
		\item if $2s < d$, then $\tilde H^s(\Omega)$ embeds in $L^q(\Omega)$ continuously for any $q \in [1,2^*]$ and compactly for any $q \in [1,2^*)$, with $2^* = 2d/(d-2s)$;
		\item if $2s = d$, then $\tilde H^s(\Omega)$ embeds in $L^q(\Omega)$ continuously for any $q \in [1,\infty)$ and compactly for any $q \in [1,2]$;
		\item if $2s > d$, then $\tilde H^s(\Omega)$ embeds continuously in $C^{0,\alpha}(\Omega)$ with $\alpha = (2s-d)/2$.
	\end{itemize}
\end{theorem}
\medskip

\section{A variational scheme for the fractional wave equation}\label{sec:free}

In this section, as a first step towards obstacle problems, we extend to the fractional wave equation a time-disretized variational scheme which traces back to Rothe \cite{Ro30} and since then has been extensively applied to many different hyperbolic type problems, see e.g. \cite{Ta94, Om97, SvadlenkaOmata08, DaLa11}.

Let $\Omega \subset \R^d$ be an open bounded domain with Lipschitz boundary.
Given $u_0 \in \Hst(\Omega)$ and $v_0 \in L^2(\Omega)$, the problem we are interested in is the following: find $u = u(t,x)$ such that
\begin{equation}\label{eq:freewaves}
\begin{system}
& u_{tt} + (-\Delta)^s u = 0             	&\quad&\text{in } (0,T) \times \Omega                                   	\\
& u(t,x) = 0                                &\quad&\text{in } [0,T] \times (\R^d \setminus \Omega)                      \\
& u(0,x) = u_0(x)                           &\quad&\text{in } \Omega                                                    \\
& u_t(0,x) = v_0(x)                         &\quad&\text{in } \Omega                                                    \\
\end{system}
\end{equation}
where the ``boundary'' condition is imposed on the complement of $\Omega$ due to the non-local nature of the fractional operator. In particular, we look for weak type solutions of \eqref{eq:freewaves}.
\begin{definition}\label{def:weak}
We say a function
\[
u \in L^\infty(0,T; \Hst(\Omega)) \cap W^{1,\infty}(0,T;L^2(\Omega)), \quad u_{tt} \in L^\infty(0,T;H^{-s}(\Omega)),
\]
is a weak solution of \eqref{eq:freewaves} if
\begin{equation}\label{eq:eqweak}
\int_{0}^T \int_\Omega u_{tt}(t) \phi(t) \,dxdt + \int_{0}^T [ u(t), \phi(t) ]_{s} \, dt = 0
\end{equation}
for all $\phi \in L^{1}(0,T;\Hst(\Omega))$ and the initial conditions are satisfied in the following sense:
\begin{equation}\label{eq:u0free}
\lim_{h\to0^+} \frac1h \int_0^h\left( ||u(t)-u_0||_{L^2(\Omega)}^2 + [u(t)-u_0]_{s}^2 \right)dt = 0
\end{equation}
and
\begin{equation}\label{eq:v0free}
\lim_{h\to0^+} \frac1h \int_0^h||u_t(t)-v_0||_{L^2(\Omega)}^2\,dt = 0.
\end{equation}
\end{definition}

\medskip

\noindent The aim of this section is then to prove the next theorem.
\begin{theorem}\label{thm:main1}
	There exists a weak solution of the fractional wave equation \eqref{eq:freewaves}.
\end{theorem}

The existence of a such a weak solution will be proved by means of an implicit variational scheme based on the idea of minimizing movements \cite{ambrosio1995minimizing} introduced by De Giorgi, elsewhere known also as the discrete Morse semiflow approach or Rothe's scheme \cite{Ro30}.

\subsection{Approximating scheme}
For any $n > 0$ let $\tau_n = T/n$, $u_{-1}^n = u_0 - \tau_n v_0$, and $u_0^n = u_0$ (conventionally we intend $v_0(x) = 0$ for $x \in \R^d \setminus \Omega$). For any $0 < i \leq n$, given $u^n_{i-2}$ and $u^n_{i-1}$, define
\begin{equation}\label{eq:scheme}
u_i^n = \arg \min_{u \in \Hst(\Omega)} J_i^n(u) = \arg \min_{u \in \Hst(\Omega)} \left[ \int_\Omega \frac{\abs{u-2u_{i-1}^n+u_{i-2}^n}^2}{2\tau_n^2}\,dx + \frac12 [u]_s^2 \right].
\end{equation}
Each $u^i_n$ is well defined: indeed, existence of a minimizer can be obtained via the direct method of the calculus of variations while uniqueness follows from the strict convexity of the functional $J_i^n$. Each minimizer $u_i^n$ can be characterize in the following way: take any test function $\phi \in \Hst(\Omega)$, then, by minimality of $u_i^n$ in $\Hst(\Omega)$, one has
\[
\frac{d}{d\varepsilon} J_i^n(u_i^n+\varepsilon \phi) |_{\varepsilon=0} = 0,
\]
which rewrites as
\begin{equation}\label{eq:ELnoobstacle}
\int_{\Omega} \frac{u_i^n-2u_{i-1}^n+u_{i-2}^n}{\tau_n^2}\,\phi \,dx + [u_i^n , \phi]_{s} = 0 \quad \text{for all } \phi \in \Hst(\Omega).
\end{equation}
We define the piecewise constant and piecewise linear interpolation in time of the sequence $\{u_i^n\}_i$ over $[-\tau_n,T]$ as follows: let $t_i^n = i\tau_n$, then the piecewise constant interpolant is given by
\begin{equation}\label{eq:uhbar}
\bar{u}^n(t,x) =
\begin{system}
& u_{-1}^n(x) &\quad &t=-\tau_n             \\
& u_i^n(x) &\quad &t \in (t_{i-1}^n,t_i^n], \\
\end{system}
\end{equation}
and the piecewise linear one by
\begin{equation}\label{eq:uh}
u^n(t,x) =
\begin{system}
& u_{-1}^n(x)                                              					&\quad &t=-\tau_n                   \\
& \frac{t-t_{i-1}^n}{\tau_n}u_i^n(x) + \frac{t_i^n-t}{\tau_n}u_{i-1}^n(x) 	&\quad &t \in (t_{i-1}^n,t_i^n].  	\\
\end{system}
\end{equation}
Define $v_i^n = (u_i^n-u_{i-1}^n)/\tau_n$, $0\leq i \leq n$, and let $v^n$ be the piecewise linear interpolation over $[0,T]$ of the family $\{v_i^n\}_{i=0}^n$, defined similarly to \eqref{eq:uh}. Taking the variational characterization \eqref{eq:ELnoobstacle} and integrating over $[0,T]$ we obtain
\[
\int_{0}^T \int_\Omega \left( \frac{u^n_t(t) - u^n_t(t-\tau_n)}{\tau_n} \right) \phi(t) \,dxdt + \int_{0}^T [ \bar{u}^n(t), \phi(t) ]_{s} \, dt = 0
\]
for all $\phi \in L^1(0,T;\Hst(\Omega))$, or equivalently
\begin{equation}\label{eq:ELn}
\int_{0}^T \int_\Omega v_t^n(t) \phi(t) \,dxdt + \int_{0}^T [ \bar{u}^n(t), \phi(t) ]_{s} \, dt = 0.
\end{equation}
The idea is now to pass to the limit $n \to \infty$ and prove, using \eqref{eq:ELn}, that the approximations $u^n$ and $\bar{u}^n$ converge to a weak solution $u$ of \eqref{eq:freewaves}. For doing so the main tool is the following estimate.
\begin{prop}[Key estimate]\label{prop:keyestimate}
	The approximate solutions $\bar{u}^n$ and $u^n$ satisfy
	\[
	\norm{ u_t^n(t) }_{ L^2(\Omega) }^{ 2 } + [ \bar{u}^n(t) ]_{s}^{ 2 } \leq C(u_0, v_0)
	\]
	for all $t \in [0,T]$, with $C(u_0, v_0)$ a constant independent of $n$.
\end{prop}

\begin{proof}
	For each fixed $i \in \{1,\dots,n\}$ consider equation \eqref{eq:ELnoobstacle} with $\phi = u_{i-1}^n-u_i^n$, so that we have
	\[
	\begin{aligned}
	0 &= \int_\Omega \frac{(u_{i}^n-2u_{i-1}^n+u_{i-2}^n)(u_{i-1}^n-u_{i}^n)}{\tau_n^2}\,dx + [u_i^n, u_{i-1}^n - u_i^n]_{s} \\
	&\leq \frac{1}{2\tau_n^2} \int_\Omega (u_{i-1}^n-u_{i-2}^n)^2 - (u_{i}^n-u_{i-1}^n)^2 \, dx + \frac{1}{2} ( [u_{i-1}^n]_{s}^{ 2 } - [u_i^n]_{s}^{ 2 } ),
	\end{aligned}
	\]
	where we use the fact that $b(a-b) \leq \frac12 (a^2-b^2)$.	Summing for $i = 1,\dots,k$, with $1 \leq k \leq n$, we get
	\[
	\begin{aligned}
	\left\lVert\frac{u_k^n-u_{k-1}^n}{\tau_n}\right\rVert_{ L^2(\Omega) }^{ 2 } + [u_k^n]_{s}^{ 2 } &\leq \frac{1}{\tau_n^2} \norm{ u_0-u_{-1}^n }_{ L^2(\Omega) }^{ 2 } + [u_0]_{s}^{ 2 } \\
	&= ||v_0||_{L^2(\Omega)}^2 + [u_0]_{s}^{ 2 }.
	\end{aligned}
	\]
	The result follows by the very definition of $u^n$ and $\bar{u}^n$.
\end{proof}

\begin{remark}\label{rem:energy}{\rm
	Given a weak solution $u$ of \eqref{eq:freewaves} we can speak of the energy quantity
	\[
	E(t) = ||u_t(t)||_{L^2(\Omega)}^2 + [u(t)]_s^2.
	\]
	One can easily see by an approximation argument that $E$ is conserved throughout the evolution and, as a by-product of the last proof, we see that also the energy of our approximations is at least non-increasing, i.e. $E_i^n \leq E_{i-1}^n$, where $E_i^n = E(u^n(t_i^n)) = ||v_i^n||_{L^2(\Omega)}^2 + [u_i^n]_{s}^2$. Furthermore we also remark that we cannot improve this estimate, meaning that generally speaking the given approximations $u^n$ are not energy preserving.
}
\end{remark}

Thanks to Proposition \ref{prop:keyestimate}, we can now prove convergence of the $u^n$.
\begin{prop}[Convergence of $u^n$]\label{prop:convun}
	There exists a subsequence of steps $\tau_n \to 0$ and a function $u \in L^\infty(0,T;\Hst(\Omega)) \cap W^{1,\infty}(0,T;L^2(\Omega))$, with $u_{tt} \in L^\infty(0,T;H^{-s}(\Omega))$, such that
	\[
	\begin{aligned}
	&u^n \to u &\text{ in }& C^0([0,T];L^2(\Omega)) \\
	&u_t^n \rightharpoonup^* u_t &\text{ in }& L^\infty(0,T;L^2(\Omega)) \\
	&u^n(t) \rightharpoonup u(t) &\text{ in }& \Hst(\Omega) \text{ for any } t \in [0,T]. \\
	\end{aligned}
	\]
\end{prop}

\begin{proof}
	
	From Proposition \ref{prop:keyestimate} it follows that
	\begin{equation}
	u_t^n(t) \text{ and } v^n(t) \text{ are bounded in } L^2(\Omega) \text{ uniformly in $t$ and $n$,}
	\end{equation}
	\begin{equation}\label{eq:boundhsdot}
	u^n(t) \text{ is bounded in the } [\cdot]_s \text{ semi-norm uniformly in $t$ and $n$.}
	\end{equation}
	Observe now that $u^n(\cdot,x)$ is absolutely continuous on $[0,T]$; thus, for all $t_1, t_2 \in [0,T]$ with $t_1 < t_2$, we have
	\[
	\begin{aligned}
	||u^n(t_2,\cdot) - u^n(t_1,\cdot)||_{L^2(\Omega)} &= \left( \int_{\Omega} \left( \int_{t_1}^{t_2} u_t^n(t,x) \,dt \right)^2\,dx  \right)^\frac12 \\
	&\leq \left( \int_{t_1}^{t_2} ||u_t^n(t,\cdot)||_{L^2(\Omega)}^2 \,dt \right)^\frac12 (t_2-t_1)^\frac12 \leq C(t_2-t_1)^\frac12,
	\end{aligned}
	\]
	where we made use of the H\"{o}lder's inequality and of Fubini's Theorem.
	This estimate yields
	\begin{equation}\label{eq:boundl2}
	u^n(t) \text{ is bounded in } L^2(\Omega) \text{ uniformly in $t$ and $n$},
	\end{equation}
	\begin{equation}\label{eq:equicont}
	u^n \text{ is equicontinuous in } C^0([0,T];L^2(\Omega)).
	\end{equation}
	From \eqref{eq:ELn}, using \eqref{eq:boundl2} and \eqref{eq:boundhsdot}, we can also deduce that $v_t^n(t)$ is bounded in $H^{-s}(\Omega)$ uniformly in $t$ and $n$. All together we have 
	\begin{equation}\label{eq:unbounds}
	u^n \text{ is bounded in } W^{1,\infty} (0,T;L^2(\Omega)) \text{ and in } L^\infty(0,T;\Hst(\Omega)),
	\end{equation}
	\begin{equation}\label{eq:vnbounds}
	v^n \text{ is bounded in } L^\infty(0,T;L^2(\Omega)) \text{ and in } W^{1,\infty}(0,T;H^{-s}(\Omega)).
	\end{equation}
	Thanks to \eqref{eq:equicont}, \eqref{eq:unbounds} and \eqref{eq:vnbounds} there exists a function $u \in L^\infty(0,T;\Hst(\Omega)) \cap W^{1,\infty}(0,T;L^2(\Omega)) \cap C^0([0,T];L^2(\Omega))$ such that
	\[
	\begin{aligned}
		&u^n \to u &\text{ in }& C^0([0,T];L^2(\Omega)) \\
		&u_t^n \rightharpoonup^* u_t &\text{ in }& L^\infty(0,T;L^2(\Omega)) \\
		&u^n(t) \rightharpoonup u(t) &\text{ in }& \Hst(\Omega) \text{ for any } t \in [0,T] \\
	\end{aligned}
	\]
	and there exists $v \in W^{1,\infty}(0,T;H^{-s}(\Omega))$ such that
	\[
	v^n \rightharpoonup^* v \text{ in } L^\infty(0,T;L^2(\Omega)) \quad\text{and}\quad v^n \rightharpoonup^* v \text{ in } W^{1,\infty}(0,T;H^{-s}(\Omega)).
	\]
	
	As one would expect $v(t) = u_t(t)$ as elements of $L^2(\Omega)$ for a.e. $t \in [0,T]$: indeed, for $t \in (t_{i-1}^n,t_i^n]$ and $\phi \in \Hst(\Omega)$, we have by construction $u_t^n(t) = v^n(t^n_i)$, and so
	\[
	\begin{aligned}
		\int_{\Omega} (u_t^n(t) - v^n(t))\phi\,dx &= \int_{\Omega} (v^n(t_i^n) - v^n(t))\phi\,dx = \int_{\Omega} \left(\int_{t}^{t^n_i} v_t^n(s)\,ds\right)\phi\,dx \\
		&\leq \tau_n ||v_t^n||_{L^\infty(0,T;H^{-s}(\Omega))} ||\phi||_{H^s(\R^d)}
	\end{aligned}
	\]
	which implies, for any $\psi(t,x) = \phi(x)\eta(t)$ with $\phi \in \Hst(\Omega)$ and $\eta \in C^1_0([0,T])$, that
	\[
	\begin{aligned}
	&\int_0^T \left[ \int_\Omega (u_t(t)-v(t))\phi\,dx\right]\eta(t) \,dt = \int_0^T\int_\Omega (u_t(t)-v(t))\psi \,dxdt \\
	&= \lim_{n\to\infty} \int_0^T\int_\Omega (u^n_t(t)-v^n(t))\psi \,dxdt = \lim_{n\to\infty} \int_0^T\left[\int_\Omega (u^n_t(t)-v^n(t))\phi\,dx\right]\eta(t)\,dt \\
	&\leq \lim_{n\to\infty} \tau_n T ||v_t^n||_{L^\infty(0,T;H^{-s}(\Omega))} ||\phi||_{H^s(\R^d)} ||\eta||_{\infty} = 0.
	\end{aligned}
	\]
	Hence we have
	\[
	\int_\Omega (u_t(t)-v(t))\phi\,dx = 0 \quad \text{ for all }\phi\in \Hst(\Omega) \text{ and a.e. } t \in [0,T],
	\]
	which yields the sought for conclusion. Thus, $v_t = u_{tt}$ and
	$
	u_{tt} \in L^\infty(0,T;H^{-s}(\Omega)).
	$
\end{proof}

\begin{prop}[Convergence of $\bar{u}^n$]\label{prop:convunbar}
	Let $u$ be the limit function obtained in Proposition \ref{prop:convun}, then
	\[
	\bar{u}^n \weakstar u \text{ in } L^\infty(0,T;\Hst(\Omega)).
	\]
\end{prop}

\begin{proof}
	By definition we have
	\[
	\begin{aligned}
	&\sup_{t \in [0,T]} \int_\Omega \abs{u^n(t,x)-\bar{u}^n(t,x)}^2 \, dx = \sum_{i=1}^n \sup_{t \in [t_{i-1}^n,t_i^n]} (t-t_i^n)^2 \int_{\Omega} (v_i^n)^2\,dx \\
	&\leq \tau_n^2 \sum_{i=1}^{n} ||v_i^n||_{L^2(\Omega)}^2 \leq C\tau_n
	\end{aligned}
	\]
	which implies $\bar{u}^n \to u$ in $L^\infty(0,T;L^2(\Omega))$. Furthermore, taking into account Proposition \ref{prop:keyestimate}, $\bar{u}^n(t)$ is bounded in $\Hst(\Omega)$ uniformly in $t$ and $n$, so that we have $\bar{u}^n \weakstar u$ in $L^\infty(0,T;\Hst(\Omega))$ and, as it happens for $u^n$, $\bar{u}^n(t) \weak u(t)$ in $\Hst(\Omega)$ for any $t \in [0,T]$.
\end{proof}

We can now pass to the limit in \eqref{eq:ELn} to prove $u$ to be a weak solution, thus proving Theorem \ref{thm:main1}.

\begin{proof}[Proof of Theorem \ref{thm:main1}]
	The limit function $u$ obtained in Proposition \ref{prop:convun} is a weak solution of \eqref{eq:freewaves}. Indeed, for each $n>0$, by \eqref{eq:ELn} one has
	\[
	\int_{0}^T \int_\Omega v_t^n(t) \phi(t) \,dxdt + \int_{0}^T [ \bar{u}^n(t), \phi(t) ]_{s} \, dt = 0
	\]
	for any $\phi \in L^1(0,T;\Hst(\Omega))$. Passing to the limit as $n \to \infty$, using Propositions \ref{prop:convun} and \ref{prop:convunbar}, we immediately get
	\[
	\int_{0}^T \int_\Omega u_{tt}(t) \phi(t) \,dxdt + \int_{0}^T [ u(t), \phi(t) ]_{s} \, dt = 0.
	\]
	Regarding the initial conditions \eqref{eq:u0free} and \eqref{eq:v0free} it suffices to prove that, if $t_k \to 0$ are Lebesgue points for both $t \mapsto ||u_t(t)||_{L^2(\Omega)}^2$ and $t \mapsto [u(t)]_{s}^2$, then
	\begin{equation}\label{eq:normconv}
	[u(t_k)]_{s}^2 \to [u_0]_{s}^2 \quad \text{and} \quad ||u_t(t_k)||_{L^2(\Omega)}^2\to ||v_0||_{L^2(\Omega)}^2.
	\end{equation}
	From the fact that $u_t \in W^{1,\infty}(0,T;H^{-s}(\Omega))$ we have $u_t(t_k) \to v_0$ in $H^{-s}(\Omega)$ and, since $u_t(t_k)$ is bounded in $L^2(\Omega)$ and $\Hst(\Omega) \subset L^2(\Omega)$ is dense, we also have $u_t(t_k) \weak v_0$ in $L^2(\Omega)$. On the other hand $u(t_k) \to u(0) = u_0$ strongly in $L^2(\Omega)$ because $u \in C^0([0,T];L^2(\Omega))$ and, being $u(t_k)$ bounded in $\Hst(\Omega)$, $u(t_k) \weak u(0)$ in $\Hst(\Omega)$ and $[u_0]_{s} \leq \liminf_k [u(t_k)]_{s}$. To prove \eqref{eq:normconv} it suffices to observe that
	\[
	\limsup_{k\to\infty} \left([u(t_k)]_{s}^2 + ||u_t(t_k)||_{L^2(\Omega)}^2\right) \leq [u_0]_{s}^2 + ||v_0||_{L^2(\Omega)}^2
	\]
	by energy conservation.
\end{proof}

\section{The obstacle problem}\label{sec:obstacle}

In this section we switch our focus to hyperbolic obstacle problems for the fractional Laplacian. We will see how a weak solution can be obtained by means of a slight modification of the previously presented scheme, whose core idea has already been used in other obstacle type problems (for example in \cite{Ki09, NoOk15}).

As above, let $\Omega \subset \R^d$ be an open bounded domain with Lipschitz boundary and consider $g \colon \Omega \to \R$, with
\[
g \in C^0(\bar{\Omega}), \quad g<0 \text{ on } \partial \Omega.
\]
We are still interested in a non-local wave type dynamic like the one of equation \eqref{eq:freewaves}, where now we require the solution $u$ to lay above $g$: this way $g$ can be interpreted as a physical obstacle that our solution cannot go below. Consider then an initial datum
\[
u_0 \in \Hst(\Omega), \quad u_0 \geq g \text{ a.e. in } \Omega,
\]
and $v_0 \in L^2(\Omega)$.
Equation \eqref{eq:freewaves}, with the addition of the obstacle $g$, reads as follows: find a function $u = u(t,x)$ such that
\begin{equation}\label{eq:obstaclewaves}
\begin{system}
& u_{tt} + (-\Delta)^s u \geq 0             &\quad&\text{in } (0,T) \times \Omega 						\\
& u(t,\cdot) \geq g                         &\quad&\text{in } [0,T] \times \Omega                       \\
& (u_{tt} + (-\Delta)^s u)(u-g) = 0         &\quad&\text{in } (0,T) \times \Omega 						\\
& u(t,x) = 0                                &\quad&\text{in } [0,T] \times (\R^d \setminus \Omega)      \\
& u(0,x) = u_0(x)                           &\quad&\text{in } \Omega                                    \\
& u_t(0,x) = v_0(x)                         &\quad&\text{in } \Omega                                    \\
\end{system}
\end{equation}
In this system the function $u$ is required to be an obstacle-free solution whenever away from the obstacle, where $u-g>0$, while we only require a variational inequality (first line) when $u$ touches $g$. The main difficulty in \eqref{eq:obstaclewaves} is the treatment of contact times: the previous system does not specify what kind of behavior arises at contact times, leaving us free to choose between ``bouncing'' solutions, the profile hits the obstacle and bounces back with a fraction of the previous velocity (see e.g. \cite{PaoliSchatzman02I}), and an ``adherent'' solution, the profile hits the obstacle and stops (this way we dissipate energy). The definition of weak solution we are going to consider includes both of these cases.
\begin{definition}\label{def:weakobst}
We say a function $u = u(t,x)$ is a weak solution of \eqref{eq:obstaclewaves} if
\begin{enumerate}
	\item $u \in L^\infty(0,T; \Hst(\Omega)) \cap W^{1,\infty}(0,T;L^2(\Omega))$ and $u(t,x) \geq g(x)$ for a.e. $(t,x) \in (0,T)\times \Omega$;
	\item there exist weak left and right derivatives $u_t^{\pm}$ on $[0,T]$ (with appropriate modifications at endpoints);
	\item for all $\phi \in W^{1,\infty}(0,T;L^2(\Omega)) \cap L^1(0,T;\Hst(\Omega))$ with $\phi \geq 0$, $\text{spt}\,\phi \subset [0,T)$, we have
	\[
	-\int_{0}^{T} \int_{\Omega} u_t\phi_t \, dxdt + \int_{0}^{T} [u, \phi]_{s} \, dt - \int_\Omega v_0\,\phi(0) \, dx \geq 0
	\]
	\item the initial conditions are satisfied in the following sense
	\[
	u(0,\cdot) = u_0, \quad \int_\Omega (u_t^+(0)-v_0)(\phi-u_0) \, dx \geq 0 \quad \forall \phi \in \Hst(\Omega), \phi \geq g.
	\]
\end{enumerate}
\end{definition}

\medskip
\noindent Within this framework we can partially extend the construction presented in the previous section so as to prove existence of a weak solution.

\begin{theorem}\label{thm:main2}
	There exists a weak solution $u$ of the hyperbolic obstacle problem \eqref{eq:obstaclewaves}, and $u$ satisfies the energy inequality
	\begin{equation}\label{eq:eneine}
	||u_t^\pm(t)||_{L^2(\Omega)}^2 + [u(t)]_{s}^2 \leq ||v_0||_{L^2(\Omega)}^2 + [u_0]_{s}^2 \quad \text{for a.e. }t\in[0,T].
	\end{equation}
\end{theorem}

\medskip

We remark here that this definition of weak solution is weaker than the one proposed in \cite{Maruo85, Eck05}, in which the authors construct a solution to \eqref{eq:obstaclewaves} as a limit of (energy preserving) solutions $u^n$ of regularized systems, where the constraint $u^n \geq g$ is turned into a penalization term in the equation. Furthermore, up to our knowledge, the problem of the existence of an energy preserving weak solution to \eqref{eq:obstaclewaves} is still open: one would expect the limit function in \cite{Maruo85, Eck05} to be the best known candidate, while a partial result for concave obstacles in $1$d was provided by Schatzman in \cite{Schatzman80}.

\subsection{Approximating scheme}\label{subsec:appschemeobtacle}
The idea is to replicate the scheme presented in Section \ref{sec:free} for the obstacle-free dynamic: define
\[
K_g = \{ u \in \Hst(\Omega) \,|\, u \geq g \text{ a.e. in } \Omega \}
\]
and, for any $n > 0$, let $\tau_n = T/n$. Define $u_{-1}^n = u_0 - \tau_n v_0$ and $u_0^n = u_0$, and construct recursively the family of functions $\{u_i^n\}_{i=1}^n \subset \Hst(\Omega)$ as
\[
u_i^n = \arg \min_{u \in K_g} J_i^n(u),
\]
with $J_i^n$ defined as in \eqref{eq:scheme}. Notice how the minimization is now over functions $u \geq g$ in $\Omega$ so that to respect the additional constraint introduced by the obstacle. Since $K_g \subset \Hst(\Omega)$ is convex, existence and uniqueness of each $u_i^n$ can be proved by means of standard arguments. Regarding the variational characterization of each minimizer $u_i^n$, we cannot take arbitrary variations $\phi \in \Hst(\Omega)$ (we may end up exiting the feasible set $K_g$), and so we need to be more careful: we take any test $\phi \in K_g$ and consider the function $(1-\varepsilon)u_i^n + \varepsilon \phi$, which belongs to $K_g$ for any sufficiently small positive $\varepsilon$. Thus, since $u_i^n$ minimizes $J_i^n$, we have the following inequality
\[
\frac{d}{d\varepsilon} J_i^n(u_i^n+\varepsilon (\phi-u_i^n)) |_{\varepsilon=0} \geq 0,
\]
which rewrites as
\begin{equation}\label{eq:vardis}
\int_{\Omega} \frac{u_i^n-2u_{i-1}^n+u_{i-2}^n}{\tau_n^2}(\phi-u_i^n)\,dx + [u_i^n , \phi-u_i^n]_{s} \geq 0 \quad \text{for all } \phi \in K_g.
\end{equation}
In particular, since every $\phi \geq u_i^n$ is an admissible test function, we also have
\begin{equation}\label{eq:vardis_simple}
\int_{\Omega} \frac{u_i^n-2u_{i-1}^n+u_{i-2}^n}{\tau_n^2}\phi\,dx + [u_i^n , \phi]_{s} \geq 0 \quad \text{for all } \phi \in \Hst(\Omega), \phi \geq 0.
\end{equation}
We define $\bar{u}^n$ and $u^n$ as, respectively, the piecewise constant and the piecewise linear interpolation in time of $\{u_i^n\}_i$ (as in \eqref{eq:uh}, \eqref{eq:uhbar}), and $v^n$ as the piecewise linear interpolant of velocities $v_i^n = (u_i^n-u_{i-1}^n)/\tau_n$, $0\leq i \leq n$. Using \eqref{eq:vardis_simple}, the analogue of \eqref{eq:ELnoobstacle} takes the following form
\[
\int_{0}^T \int_\Omega \left( \frac{u^n_t(t) - u^n_t(t-\tau_n)}{\tau_n} \right) \phi(t) \,dxdt + \int_{0}^T [ \bar{u}^n(t), \phi(t) ]_{s} \, dt \geq 0
\]
for all $\phi \in L^1(0,T;\Hst(\Omega))$, $\phi(t,x) \geq 0$ for a.e. $(t,x) \in (0,T)\times \Omega$.

In view of a convergence result, we observe that the same energy estimate of Proposition \ref{prop:keyestimate} extends to this new context: for any $n>0$, we have
\[
\norm{ u_t^n(t) }_{ L^2(\Omega) }^{ 2 } + [ \bar{u}^n(t) ]_{s}^{ 2 } \leq C(u_0, v_0)
\]
for all $t \in [0,T]$, with $C(u_0, v_0)$ a constant independent of $n$. The exact same proof of Proposition \ref{prop:keyestimate} applies: just observe that, taking $\phi = u_{i-1}^n$ in \eqref{eq:vardis}, one gets
\[
0 \leq \int_\Omega \frac{(u_{i}^n-2u_{i-1}^n+u_{i-2}^n)(u_{i-1}^n-u_{i}^n)}{\tau_n^2}\,dx + [u_i^n, u_{i-1}^n - u_i^n]_{s}
\]
and then the rest follows. Convergence of the interpolants is then a direct consequence. 

\begin{prop}[Convergence of $u^n$ and $\bar{u}^n$, obstacle case]\label{prop:convunobstacle}
	There exists a subsequence of steps $\tau_n \to 0$ and a function $u \in L^\infty(0,T;\Hst(\Omega)) \cap W^{1,\infty}(0,T;L^2(\Omega))$ such that
	\[
	\begin{aligned}
	&u^n \to u \text{ in } C^0([0,T];L^2(\Omega)), &\quad& \bar{u}^n \weakstar u \text{ in } L^\infty(0,T;\Hst(\Omega)), \\
	&u_t^n \rightharpoonup^* u_t \text{ in } L^\infty(0,T;L^2(\Omega)), &\quad& u^n(t) \rightharpoonup u(t) \text{ in } \Hst(\Omega) \text{ for any } t \in [0,T], \\
	\end{aligned}
	\]
	and furthermore $u(t,x)\geq g(x)$ for a.e. $(t,x) \in [0,T]\times\Omega$.
\end{prop}

\begin{proof}
To obtain the existence of $u$ and all the convergences we can repeat the first half of the proof of Proposition \ref{prop:convun} and the proof of Proposition \ref{prop:convunbar}. The fact that $u(t,x)\geq g(x)$ for a.e. $(t,x) \in [0,T]\times\Omega$ is a direct consequence of the fact that $u_i^n \in K_g$ for all $n$ and $0\leq i\leq n$.
\end{proof}

The missing step with respect to the obstacle-free dynamic is that generally speaking $u_{tt} \notin L^\infty(0,T;H^{-s}(\Omega))$. The cause of such a behavior is clear already in $1$d: suppose the obstacle to be $g=0$ and imagine a flat region of $u$ moving downwards at a constant speed; when this region reaches the obstacle the motion cannot continue its way down (we need to stay above $g$) and so the velocity must display an instantaneous and sudden change in a region of non-zero measure (within our scheme the motion stops on the obstacle and velocity drops to $0$ on the whole contact region). Due to this possible behavior of $u_t$, we cannot expect $u_{tt}$ to posses the same regularity as in the obstacle-free case. Nevertheless, such discontinuities in time of $u_t$ are somehow controllable and we can still provide some sort of regularity results, which are collected in the following propositions. 

\begin{prop}\label{prop:FBV}
	Let $u$ be the weak limit obtained in Proposition \ref{prop:convunobstacle} and, for any fixed $0\leq \phi\in \Hst(\Omega)$, let $F \colon [0,T] \to \R$ be defined as
	\begin{equation}\label{eq:F}
	F(t) = \int_{\Omega}u_t(t)\phi\,dx.
	\end{equation}
	Then $F \in BV(0,T)$ and, in particular, $u^n_t(t) \weak u_t(t)$ in $L^2(\Omega)$ for a.e. $t \in [0,T]$.
\end{prop}

\begin{proof}
Let us fix $\phi \in \Hst(\Omega)$ with $\phi \geq 0$, and consider the functions $F^n \colon [0,T] \to \R$ defined as
\begin{equation}\label{eq:Fn}
F^n(t) = \int_{\Omega}^{} u_t^n(t)\phi \,dx.
\end{equation}
Observe that $||F^n||_{L^1(0,T)}$ is uniformly bounded because $u_t^n$ is bounded in $L^2(\Omega)$ uniformly in $n$ and $t$. Furthermore, for every fixed $n > 0$ and $0\leq i\leq n$, we deduce from \eqref{eq:vardis_simple} that
\begin{equation}\label{eq:ELn2}
\Bigg\lvert \int_{\Omega}^{} (v_i^n - v_{i-1}^n)\phi\,dx \Bigg\rvert - \int_\Omega (v_i^n-v_{i-1}^n)\phi\,dx \leq \tau_n\abs{[u_i^n,\phi]_{s}} - \tau_n[u_i^n,\phi]_{s}.
\end{equation}
Summing over $i = 1,\dots,n$ and using Proposition \ref{prop:keyestimate}, we get
\[
\begin{aligned}
\sum_{i=1}^{n} &\Bigg\lvert \int_{\Omega} (v_i^n - v_{i-1}^n)\phi\,dx \Bigg\rvert \leq \int_\Omega v_{n}^n\phi \, dx - \int_\Omega v_0\phi\,dx + \sum_{i=1}^{n} \tau_n\abs{[u_i^n,\phi]_{s}} - \sum_{i=1}^{n} \tau_n[u_i^n,\phi]_{s} \\ &\leq ||v_n^n||_{L^2(\Omega)} ||\phi||_{L^2(\Omega)} + ||v_0||_{L^2(\Omega)} ||\phi||_{L^2(\Omega)} + 2\tau_n \sum_{i=1}^{n} \abs{[u_i^n,\phi]_{s}} \\ &\leq ||v_n^n||_{L^2(\Omega)} ||\phi||_{L^2(\Omega)} + ||v_0||_{L^2(\Omega)} ||\phi||_{L^2(\Omega)} + 2\tau_n \sum_{i=1}^{n} [u_i^n]_{s} [\phi]_{s} \\ &\leq C||\phi||_{H^s(\R^d)}
\end{aligned}
\]
with $C$ independent of $n$. Thus, $\{F^n\}_n$ is uniformly bounded in $BV(0,T)$ and by Helly's selection theorem there exists a function $\bar{F}$ of bounded variation such that $F^n(t) \to \bar{F}(t)$ for every $t \in (0,T)$.

Take now $\psi(t,x) = \phi(x) \eta(t)$ for $\eta \in C^\infty_c(0,T)$, using that $u_t^n \weakstar u_t$ in $L^\infty(0,T;L^2(\Omega))$, one has
\[
\begin{aligned}
&\int_{0}^{T}\int_{\Omega}^{} u_t(t)\psi \,dxdt = \lim_{n \to \infty} \int_{0}^{T} \int_\Omega u_t^n(t)\psi \,dxdt =
\lim_{n \to \infty} \int_{0}^{T} \int_\Omega u_t^n(t)\phi \,dx \,\eta(t) dt \\ &=\int_{0}^{T} \lim_{n \to \infty} \int_\Omega u_t^n(t)\phi \,dx \,\eta(t) \,dt = \int_{0}^{T} \bar{F}(t)\eta(t)\,dt
\end{aligned}
\]
where the passage to the limit under the sign of integral is possible due to the pointwise convergence of $F^n$ to $\bar{F}$ combined with the dominated convergence theorem. We conclude
\[
\int_{0}^{T} \left( \int_{\Omega} u_t(t)\phi \,dx -\bar{F}(t)\right) \eta(t)\,dt = 0
\]
and, by the arbitrariness of $\eta$, we have $F = \bar{F}$ for a.e. $t \in (0,T)$, which is to say $F \in BV(0,T)$. In particular,
\[
\int_{\Omega} u_t(t) \phi \,dx = F(t) = \lim_{n\to\infty} \int_{\Omega} u_t^n(t)\phi\,dx \quad \text{ for a.e. }t\in (0,T),
\]
meaning $u_t^n(t) \rightharpoonup u_t(t)$ in $L^2(\Omega)$ for almost every $t\in(0,T)$: indeed the last equality can first be extended to every $\phi \in \Hst(\Omega)$ (just decomposing $\phi=\phi^+-\phi^-$ in its positive and negative parts) and then to every $\phi \in L^2(\Omega)$ being $\Hst(\Omega) \subset L^2(\Omega)$ dense.
\end{proof}

\begin{remark}
	In the rest of this section we choose to use the ``precise representative'' of $u_t$ given by $u_t(t) = $ weak-$L^2$ limit of $u^n_t(t)$, which is then defined for all $t \in [0,T]$.
\end{remark}

\begin{prop}	
	Fix $0\leq \phi\in \Hst(\Omega)$ and let $F$ de defined as in \eqref{eq:F}. Then, for any $t \in (0,T)$, we have
	\[
	\lim_{r\to t^-} F(r) \leq \lim_{s\to t^+}F(s).
	\]
\end{prop}

\begin{proof}
	First of all we observe that the limits we are interested in exist because $F \in BV(0,T)$. Fix then $t \in (0,T)$ and let $0<r<t<s<T$. For each $n$ define $r_n$ and $s_n$ such that $r \in (t^n_{r_n-1},t^n_{r_n}]$ and $s \in (t^n_{s_n-1},t^n_{s_n}]$. If we consider the functions $F^n$ defined in \eqref{eq:Fn} and take into account \eqref{eq:ELn2}, one can see that
	\[
	\begin{aligned}
	F^n(s)-F^n(r) &= \int_{\Omega} (u_t^n(s)-u_t^n(r))\phi\,dx = \int_{\Omega} (v^n_{s_n}-v^n_{r_n})\phi \,dx \\
	&= \sum_{i=r_n+1}^{s_n} \int_{\Omega} (v^n_{i}-v^n_{i-1})\phi \,dx \geq \tau_n \sum_{i=r_n+1}^{s_n} \left( [u_i^n,\phi]_{s} - \abs{[u_i^n,\phi]_{s}} \right) \\
	&\geq -2C \tau_n (s_n-r_n) ||\phi||_{H^s(\R^d)}
	\end{aligned}
	\]
	for some positive constant $C$ independent of $n$. Since $|s-r|\geq |t^n_{s_n-1}-t^n_{r_n}|=\tau_n(s_n-1-r_n)$ we can conclude
	\[
	F^n(s)-F^n(r) \geq -2C|s-r|\cdot||\phi||_{H^s(\R^d)}-2C\tau_n||\phi||_{H^s(\R^d)}.
	\]
	Passing to the limit $n \to \infty$ we get $F(s)-F(r) \geq -2C|s-r|\cdot||\phi||_{H^s(\R^d)}$, which in turn implies the conclusion.
\end{proof}
The last result tells us that the velocity $u_t$ does not present sudden changes in regions where it is positive, accordingly with the fact that whenever we move upwards there are no obstacles to the dynamic and $u_t$ is expected to have, at least locally in time and space, the same regularity it has in the obstacle-free case.

We eventually switch prove conditions 2, 3 and 4 of our definition of weak solution, thus proving Theorem \ref{thm:main2}.

\begin{proof}[Proof of Theorem \ref{thm:main2}]
	Let $u$ be the limit function obtained in Proposition \ref{prop:convunobstacle}. We verify one by one the four conditions required in Definition \ref{def:weakobst}.

	\emph{(1.)} The first condition is verified thanks to Proposition \ref{prop:convunobstacle}.
	
	\emph{(2.)} Existence of weak left and right derivatives $u_t^{\pm}$ on $[0,T]$ follows from Proposition \ref{prop:FBV}: just observe that, for any fixed $\phi \in \Hst(\Omega)$, the function
	\[
	F(t) = \int_{\Omega}u_t(t)\phi\,dx
	\]
	is $BV(0,T)$ and thus left and right limits of $F$ are well defined for any $t \in [0,T]$. This, in turn, implies condition 2. in our definition of weak solution.
	
	\emph{(3.)} For $n > 0$ and any test function $\phi \in W^{1,\infty}(0,T;L^2(\Omega)) \cap L^1(0,T;\Hst(\Omega))$, with $\phi \geq 0$, $\text{spt}\,\phi \subset [0,T)$, we recall that
	\[
	\int_{0}^T \int_\Omega \left( \frac{u^n_t(t) - u^n_t(t-\tau_n)}{\tau_n} \right) \phi(t) \,dxdt + \int_{0}^T [ \bar{u}^n(t), \phi(t) ]_{s} \, dt \geq 0.
	\]
	Thanks to Proposition \ref{prop:convunobstacle}, we have
	\[
	\begin{aligned}
	\int_{0}^{T} [\bar{u}^n(t), \phi(t)]_{s} \, dt \to \int_{0}^{T} [u(t), \phi(t)]_{s} \, dt \quad \text{as } n \to \infty
	\end{aligned}
	\]
	while, on the other hand, we also have
	\[
	\begin{aligned}
	& \int_{0}^{T} \int_\Omega \frac{u^n_t(t)-u^n_t(t-\tau_n)}{\tau_n}\,\phi(t) \,dxdt 
	 = \int_{0}^{T-\tau_n}\int_\Omega u^n_t(t) \left( \frac{\phi(t)-\phi(t+\tau_n)}{\tau_n} \right) \, dxdt \\
	&- \int_{0}^{\tau_n} \int_\Omega \frac{v_0}{\tau_n}\,\phi(t) \, dxdt + \int_{T-\tau_n}^{T} \int_\Omega \frac{u^n_t(t)}{\tau_n}\,\phi(t) \, dxdt \\
	& \to \int_{0}^{T} \int_\Omega u_t(t)(-\phi_t(t)) \, dxdt - \int_\Omega v_0\,\phi(0) \, dx + 0 \quad \text{as } n \to \infty.
	\end{aligned}
	\]
	This proves condition 3. for weak solutions.
	
	\emph{(4.)} The fact that $u(0) = u_0$ is a direct consequence of $u^n(0)=u_0$ and of the convergence of $u^n$ to $u$ in $C^0([0,T];L^2(\Omega))$.
	We are left to check the initial condition on velocity. Suppose, without loss of generality, that the sequence $u^n$ is constructed by taking $n \in \{2^m \,: m > 0\}$ (each successive time grid is obtained dividing the previous one).  Fix then $n$ and $\phi \in K_g$, let $T^* = m\tau_n$ for $0 \leq m \leq n$ (i.e. $T^*$ is a ``grid point''). Let us evaluate
	\[
	\begin{aligned}
	& \int_0^{T^*} \int_{\Omega}^{} \frac{u_t^n(t)-u_t^n(t-\tau_n)}{\tau_n} (\phi - \bar{u}^n(t)) = \sum_{i=1}^m \int_{t_{i-1}^n}^{t_i^n} \int_\Omega \frac{u_i^n-2u_{i-1}^n+u_{i-2}^n}{\tau_n^2}(\phi-u_i^n) \\
	& = \int_{\Omega}^{} \sum_{i=1}^m \frac{u_i^n-2u_{i-1}^n+u_{i-2}^n}{\tau_n}(\phi-u_i^n) = \int_{\Omega}^{} \sum_{i=1}^m (v_i^n-v_{i-1}^n)(\phi-u_i^n) \\
	&= -\int_{\Omega}^{} v_0^n(\phi-u_1^n)\,dx + \int_\Omega v_m^n(\phi-u_m^n) \,dx + \tau_n \sum_{i=1}^{m-1} \int_\Omega v_i^nv_{i-1}^n\,dx \\
	&= -\int_{\Omega}^{} v_0(\phi-u_n(\tau_n))\,dx + \int_\Omega u_t^n(T^*)(\phi-u^n(T^*)) \,dx + \tau_n \sum_{i=1}^{m-1} \int_\Omega v_i^nv_{i-1}^n\,dx. \\
	\end{aligned}
	\]
	Using \eqref{eq:vardis} we observe that
	\[
	\int_{0}^{T^*} \int_\Omega \frac{u^n_t(t) - u^n_t(t-\tau_n)}{\tau_n} ( \phi - \bar{u}^n(t) ) \,dxdt + \int_{0}^{T^*} [ \bar{u}^n(t), \phi-\bar{u}^n(t) ]_{s} \, dt \geq 0,
	\]
	which combined with the above expression and previous estimates on $u_i^n$ and $v_i^n$ leads to
	\[
	\begin{aligned}
	&-\int_{\Omega}^{} v_0(\phi-u_n(\tau_n))\,dx + \int_\Omega u_t^n(T^*)(\phi-u^n(T^*)) \,dx \geq\\
	& -\tau_n \sum_{i=1}^{m-1} \int_\Omega v_i^nv_{i-1}^n\,dx - \tau_n \sum_{i=1}^{m} [u_i^n, \phi-u_i^n]_{s} \geq -CT^* - CT^* ||\phi||_{H^s(\R^d)}.
	\end{aligned}
	\]
	Passing to the limit as $n \to \infty$, using $u^n(\tau_n) \to u(0)$ and $u_t^n(T^*) \rightharpoonup u_t(T^*)$ (due to the use of the precise representative), we get
	\[
	\begin{aligned}
	&-\int_{\Omega}^{} v_0(\phi-u(0))\,dx + \int_\Omega u_t(T^*)(\phi-u(T^*)) \,dx \geq -CT^* - C||\phi||_{H^s(\R^d)} T^*.
	\end{aligned}
	\]
	Taking now $T^* \to 0$ along a sequence of ``grid points'' we have
	\[
	\int_\Omega (u_t^+(0)-v_0)(\phi-u(0)) \,dx \geq 0.
	\]
	And this completes the first part of the proof. We are left to prove the energy inequality \eqref{eq:eneine}.	For this, recall that from Remark \ref{rem:energy} it follows that, for all $n > 0$,
		\[
		||v^n(t)||_{L^2(\Omega)}^2 + [u^n(t)]_{s}^2 \leq ||v_0||_{L^2(\Omega)}^2 + [u_0]_{s}^2 \quad \text{for all }t\in[0,T].
		\]
		Passing to the limit as $n \to \infty$ we immediately get \eqref{eq:eneine}.
\end{proof}

We conclude this section with some remarks and observations about the solution $u$ obtained through the proposed semi-discrete convex minimization scheme in the scenario $s = 1$. First of all we identify the weak solution $u$ obtained above to be a more regular solution whenever approximations $u^n$ stay strictly above $g$.

\begin{prop}[Regions without contact]\label{prop:nocontact}
	Let $s = 1$ and, for $\delta > 0$, suppose there exists an open set $A_\delta \subset \Omega$ such that $u^n(t,x) > g(x) + \delta$ for a.e. $(t,x) \in (0,T)\times \Omega$ and for all $n > 0$. Then $u_{tt} \in L^\infty(0,T;H^{-1}(A_\delta))$ and $u$ satisfies \eqref{eq:eqweak} for all $\phi \in L^{1}(0,T;H^1_0(A_\delta))$.
\end{prop}

\begin{proof}
	Take $\phi \in H^1_0(\Omega)$ with $\textup{spt}\,\phi \subset A_\delta$. Then, for every $n$ and $0 \leq i \leq n$, the function $u^n_i + \varepsilon\phi$ belongs to $K_g$ for $\varepsilon$ sufficiently small: indeed, for $x \in A_\delta$, we have $u_i^n(x) + \epsilon \phi(x) \geq g(x) + \delta + \epsilon\phi(x) \geq g(x)$ for small $\epsilon$, regardless of the sign of $\phi(x)$. In particular, equation \eqref{eq:vardis_simple} can be written as
	\[
	\int_{\Omega} \frac{u_i^n-2u_{i-1}^n+u_{i-2}^n}{\tau_n^2}\phi\,dx + \int_\Omega \nabla u_i^n \cdot  \nabla \phi\,dx = 0 \quad \text{for all } \phi \in H^1_0(\Omega), \textup{spt}\,\phi \subset A_\delta.
	\]
	This equality allows us to carry out the second part of the proof of Proposition \ref{prop:convun}, so that, in the same notation, we can prove $v^n_t(t)$ to be bounded in $H^{-1}(A_\delta)$ uniformly in $t$ and $n$. Thus, $v \in W^{1,\infty}(0,T;H^{-1}(A_\delta))$ and
	\[
	v^n \rightharpoonup^* v \text{ in } L^\infty(0,T;L^2(A_\delta)) \quad\text{and}\quad v^n \rightharpoonup^* v \text{ in } W^{1,\infty}(0,T;H^{-1}(A_\delta)).
	\]
	Localizing everything on $A_\delta$, we can prove $v_t = u_{tt}$ in $A_\delta$ so that
	\[
	u_{tt} \in L^\infty(0,T;H^{-1}({A_\delta})),
	\]
	and equation \eqref{eq:eqweak} follows by passing to the limit as done in the proof of Theorem \ref{thm:main1} (cf. \cite{SvadlenkaOmata08, DaLa11}).
\end{proof}

\begin{remark}[One dimensional case with $s=1$]{\rm 
	In the one dimensional case and for $s=1$ the analysis boils down to the problem considered by Kikuchi in \cite{Ki09}. In this particular situation a stronger version of Proposition \ref{prop:nocontact} holds: suppose that $\Omega = [0,1]$, then for any $\phi \in C^0_0([0,T),L^2(0,1)) \cap W^{1,2}_0((0,T)\times(0,1))$ with $\textup{spt}\,\phi \subset \{ (t,x) \,: u(t,x)>0 \}$,
	\[
	-\int_{0}^{T} \int_0^1 u_t\phi_t \, dxdt + \int_{0}^{T} \int_0^1 u_x \phi_x \, dxdt - \int_0^1 v_0\,\phi(0) \, dx = 0.
	\]
}
\end{remark}

\section{Numerical implementation and open questions}
The constructive scheme presented in the previous sections can be easily used to provide a numerical simulation of the relevant dynamic, at least in the case $s = 1$ where we can employ a classical finite element discretization. However, we observe that a similar finite element approach can be extended to the fractional setting $s < 1$ following for example the pipeline described in \cite{AiGl18, AiGl17}.

Minimization of energies $J_i^n$ can be carried out by means of a piecewise linear finite element approximation in space: given a triangulation $\mathcal{T}_h$ of the domain $\Omega$ we introduce the classical space
\[
X_h^1 = \{ u_h \in C^0(\bar{\Omega}) \,:\, u_h |_K \in \mathbb{P}_1(K),\text{ for all } K \in \mathcal{T}_h \}.
\]
For $n > 1$, $0 < i \leq n$, and given $u_{i-1}^n, u_{i-2}^n \in X_h^1$, we optimize $J_i^n$ among functions in $X_h^1$ respecting the prescribed Dirichlet boundary conditions (which are local because $s=1$). We get this way a finite dimensional optimization problem for the degrees of freedom of $u_i^n$ and we solve it by a gradient descend method combined with a dynamic adaptation of the descend step size.

\begin{figure}[tbh]
	\centering
	\begin{tabular}{cc}
		\includegraphics[width=0.4\linewidth]{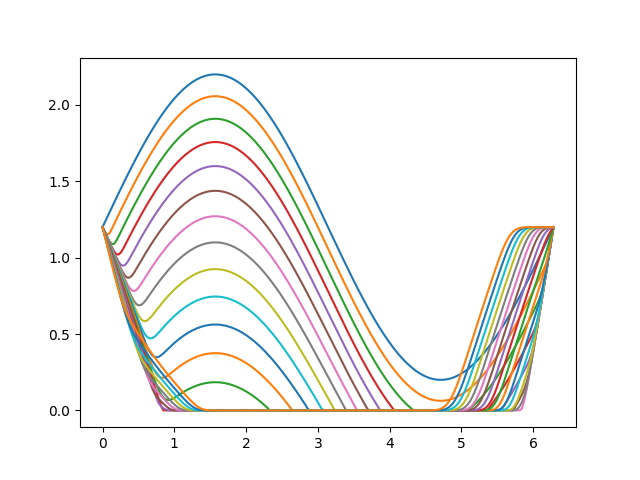}
		\includegraphics[width=0.7\linewidth]{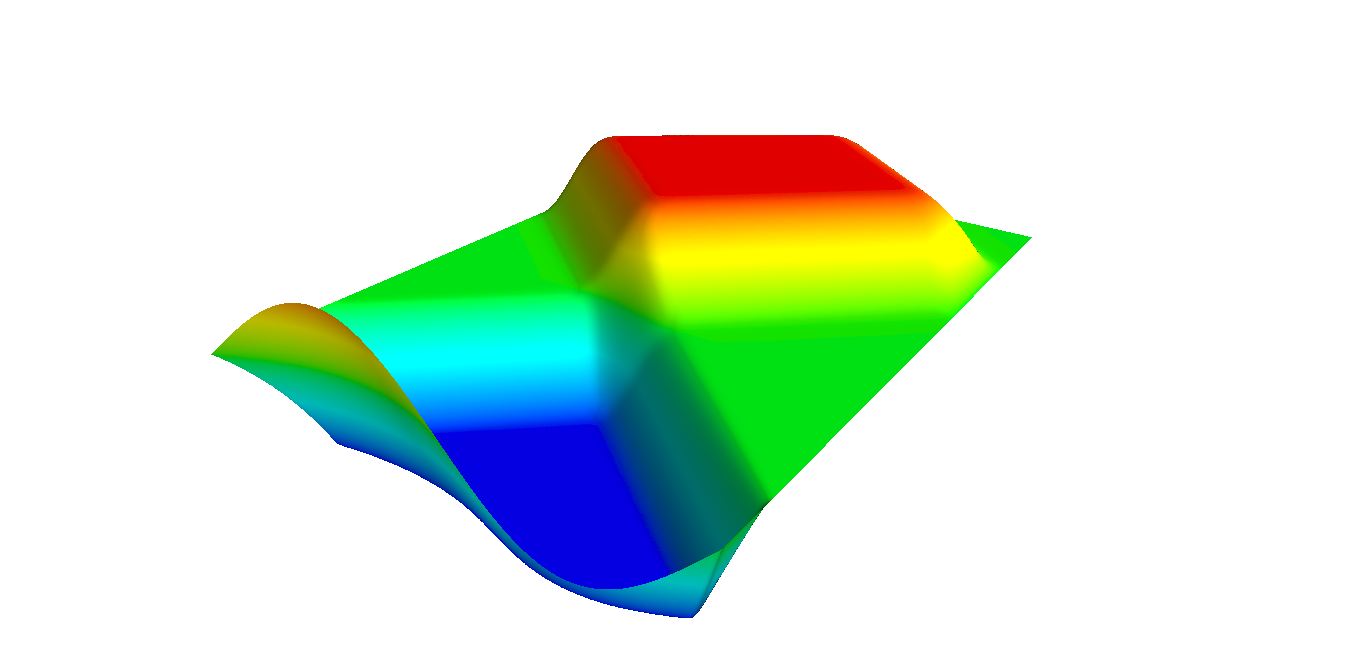}
	\end{tabular}
	\caption{Time evolution of the solution till $t = 1.5$ (left) and space-time depiction of the same evolution till $T = 10$ (right).}
	\label{fig:1d_exe}
\end{figure}

In the simulation in figure \ref{fig:1d_exe} we take $\Omega = (0,2\pi)$ and $u_0(x) = \sin(x) + 1.2$, with a constant initial velocity of $-2$ which pushes the string towards the obstacle $g=0$. The boundary conditions are set to be $u(t,0) = u(t,2\pi) = 1.2$ and the simulation is performed up to $T = 10$ using a uniform grid with $h = 2\pi/200$ and a time step $\tau = 1/100$. We can see how the profile stops on the obstacle after impact (blue region in the right picture of figure \ref{fig:1d_exe}) and how the impact causes the velocity to drop to $0$ and thus a loss of energy (as displayed in figure \ref{fig:1d_exe_E}). As soon as the profile leaves the obstacle the dynamic goes back to a classical wave dynamic and energy somehow stabilizes even if, as expected, it is not fully conserved from a discrete point of view. Due to energy dissipation at impact times, in the long run we expect the solution to never hit the obstacle again because the residual energy will only allow the profile to meet again the obstacle at $0$ speed, i.e. without any loss of energy.
\begin{figure}[tbh]
	\centering
	\begin{tabular}{cc}
		\includegraphics[width=0.4\linewidth]{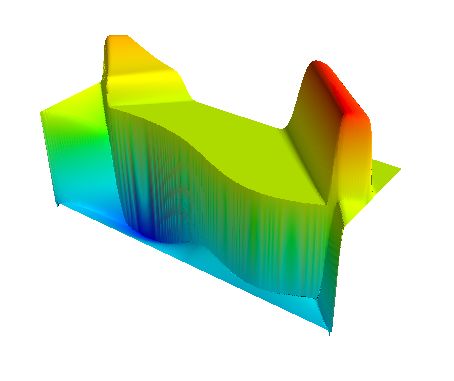} &
		\includegraphics[width=0.4\linewidth]{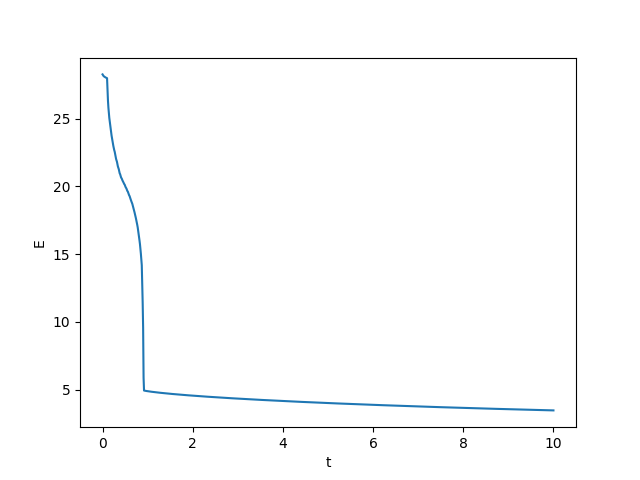}
	\end{tabular}
	\caption{Time evolution of the velocity up to $t = 2$ (left) and energy (right).}
	\label{fig:1d_exe_E}
\end{figure}
Thus, also in higher dimension, we expect the solution $u$ obtained through the proposed scheme to become an obstacle-free solution of the wave equation as soon as the energy of the system drops below a certain value, preventing this way future collisions. This can be roughly summarized in the following conjecture.
\begin{conj}[Long time behavior]\label{conj:1}
Let $s = 1$ and, given an obstacle problem in the form of equation \eqref{eq:obstaclewaves}, let $u$ be the weak solution obtained through the convex minimization approach of Section \ref{subsec:appschemeobtacle}. Then, at least for sufficiently regular obstacles $g$, there exists $\bar{t} > 0$ such that $E(u(t))$ is constant for any $t > \bar{t}$.
\end{conj}

Alongside the previous conjecture, we observe that the solution $u$ obtained here seems to be, among all possible weak solutions, the one dissipating its kinetic energy at highest rate, when colliding with the obstacle $g$, and so the one realizing the ``adherent'' behavior we mentioned before. At the same time, from the complete opposite perspective, one could ask if it is possible to revise the scheme so that to obtain energy preserving approximations $u^n$, and try to use these approximations to provide an energy preserving weak solution (maybe under suitable additional hypothesis on the obstacle).

As already observed in the introduction, the proposed method can be extended to the case of semi-linear wave equations of the type
\[
u_{tt} + (-\Delta)^s u + f(u) = 0
\]
with $f$ a suitable function, possibly non-smooth. For example, one can consider $f$ to be the (scaled) derivative of a balanced, double-well potential, e.g. $f(u) = \frac{1}{\epsilon^2}(u^3-u)$ for $\epsilon > 0$: certain solutions of that equation are intimately related to timelike minimal hypersurfaces, i.e. with vanishing mean curvature with respect to Minkowski space-time metric \cite{del2018interface, jerrard2011defects, bellettini2010time}. On the other hand, as we said in the introduction, one could also manage adhesive type dynamics assuming $f$ to be the (non-smooth) derivative of a smooth potential $\Phi$, as it is done in \cite{CocliteFlorioLigaboMaddalena17}.

We eventually observe that the proposed approximations $u^n$ can be constructed, theoretically and numerically, also for a double obstacle problem, i.e. $g(x) \leq u(t,x) \leq f(x)$ for a suitable lower obstacle $g$ and upper obstacle $f$. However, in this new context, the previous convergence analysis cannot be replicated because even the basic variational characterization \eqref{eq:vardis_simple} is generally false and a more localized analysis would be necessary. Anyhow, also in this situation one would expect the solution to behave like an obstacle-free solution after some time, as suggested in Conjecture \ref{conj:1}.

\section*{Acknowledgements}

The authors are partially supported by GNAMPA-INdAM. The second author acknowledges partial support
by the University of Pisa Project PRA 2017-18.

\bibliographystyle{plain}
\bibliography{bibliography}

\end{document}